\documentclass[reqno,12pt]{amsart}        
\usepackage{amsfonts}
\usepackage{colordvi}
\usepackage{amsmath,amssymb,amsthm} 
\usepackage{eufrak}
\usepackage{mathrsfs}
\usepackage{graphicx}
\usepackage{xypic}
\usepackage{tipa}    
\theoremstyle{plain}      
\newtheorem{thm}{Theorem}[section]     
\newtheorem{theorem}[thm]{Theorem}     
\newtheorem{corollary}[thm]{Corollary}     
\newtheorem{lemma}[thm]{Lemma}     
\newtheorem{proposition}[thm]{Proposition}

\theoremstyle{remark}      
\newtheorem{example}[thm]{Example} 
\newtheorem{remark}[thm]{Remark} 
     
\theoremstyle{definition}

\DeclareMathOperator{\Arg}{Arg}

\def\C{\mathbb{C}}
\def\K{\mathbb{K}}
\def\R{\mathbb{R}}
\def\T{\mathbb{T}}
\def\Z{\mathbb{Z}}
\renewcommand{\P}{{\mathbb P}}

\def\scrA{\mathscr{A}}

\def\scrP{\mathscr{P}}
\def\scrT{\mathscr{T}}
\def\coscrA{{\it co}\mathscr{A}}
\def\lcoscrA{\widehat{{\it co}}\mathscr{A}}
\newcommand{\nca}{{\mathit{Nco}\mathscr{A}}}
\newcommand{\lnca}{\widehat{{\mathit{Nco}}\mathscr{A}}}
\def\lsh{\widehat{{\it sh}}}

\DeclareMathOperator{\Log}{Log} 
\DeclareMathOperator{\Pl}{PL} 
\DeclareMathOperator{\supp}{supp} 

\newcommand{\ini}{{\rm in}}

\newcommand{\arc}[1]{%
  \setbox9=\hbox{#1}%
  \ooalign{\resizebox{\wd9}{\height}{\texttoptiebar{\phantom{pq}}}\cr$#1$}}
\newcommand{\barc}[1]{%
  \setbox9=\hbox{#1}%
  \ooalign{\resizebox{\wd9}{\height}{\texttoptiebar{\phantom{p'q'}}}\cr$#1$}}

\let\witi\widetilde
\newcommand{\defcolor}[1]{\Blue{#1}}
\newcommand{\demph}[1]{\defcolor{{\sl #1}}}

\begin{document}     


\title{Higher convexity of coamoeba complements}  
 
\author{Mounir Nisse}
\address{Mounir Nisse\\School of Mathematics KIAS, 87 Hoegiro Dongdaemun-gu, Seoul
130-722, South Korea.}
\email{mounir.nisse@gmail.com}
\author{Frank Sottile}
\address{Frank Sottile\\
         Department of Mathematics\\
         Texas A\&M University\\
         College Station\\
         Texas\\
         USA}
\email{sottile@math.tamu.edu}
\urladdr{www.math.tamu.edu/\~{}sottile}
\thanks{Research of Sottile is supported in part by NSF grant DMS-1001615.}
\thanks{This material is based upon work supported by the National Science 
Foundation under Grant No. 0932078 000, while Sottile was in 
residence at the Mathematical Science Research Institute (MSRI) in 
Berkeley, California, during the winter semester of 2013.}
%
%

\begin{abstract}
 We show that the complement of the closure of the coamoeba of a variety of codimension $k{+}1$ is
 $k$-convex, in the sense of Gromov and Henriques.
 This generalizes a result of Nisse for hypersurface coamoebas.
 We use this to show that the complement of the nonarchimedean coamoeba of a variety of codimension
 $k{+}1$ is $k$-convex 
\end{abstract}

\maketitle

\section*{Introduction} 

The amoeba $\scrA$ of an algebraic subvariety of a torus $(\C^*)^n$ is its image in 
$\R^n$ under the logarithmic moment map ($z\mapsto\log|z|$ in each coordinate),
and its coamoeba $\coscrA$ is its image in $(S^1)^n$ under the coordinatewise argument
map. 
Some structures of a variety are reflected in its amoeba and coamoeba.
We study relations between topological properties of
a variety and of its amoeba and coamoeba.

A fundamental property of amoebas and coamoebas of hypersurfaces is that their complements have convex connected
components. 
Gelfand, Kapranov, and Zelevinsky showed this in their monograph introducing amoebas~\cite{GKZ}.
Nisse showed this for coamoebas~\cite[Th.~5.19]{Ni09}, where the complement is
taken in the universal cover $\R^n$ of $(S^1)^n$.

Gromov~\cite[\S$\tfrac{1}{2}$]{Gromov} introduced a generalized notion of convexity for domains 
in $\R^n$.
A subset $X\subset\R^n$ is \demph{$k$-convex} if for every $(k{+}1)$-dimensional
affine subspace $L$ of $\R^n$ the map of $k$th reduced homology groups 
$\witi{H}_k(L\cap X)\to\witi{H}_k(X)$ induced by the inclusion 
$L\cap X\hookrightarrow X$ is an injection.
Then $0$-convex is equivalent to the convexity of each connected component.
Gromov shows that if $X$ is a domain with a smooth boundary $W$, then $k$-convexity of $X$ is equivalent to  
the nonegativity of $n{-}k{-}1$ of the principal curvatures of $W$.
Like ordinary convexity, he shows that higher convexity behaves well under intersection:
If $X,Y\subset\R^n$ are $k$-convex, then $X\cap Y$ is also $k$-convex.

Thus Gelfand, et al.~\cite{GKZ} showed that the complement of a hypersurface amaoeba is
$0$-convex. 
Henriques~\cite{He04} conjectured that for an algebraic subvariety of the torus of
codimension $k{+}1$, the complement $\scrA^c$ of its amoeba is $k$-convex. 
He proved a weak version, that a nonnegative class in 
$\witi{H}_k(L\cap \scrA^c)$ has nonzero image in $\witi{H}_k(\scrA^c)$.
Mikhalkin~\cite{Mikhalkin} considered a local version, showing that $\scrA$ has no supporting $k$-cap, 
which is a ball $B$ in an affine $k{+}1$-dimensional plane such that $\scrA\cap B$ is nonempty and
compact, while $\scrA\cap(B+\epsilon v)$ is empty for some $v\in\R^n$ and all sufficiently small positive
$\epsilon$.

Favorov~\cite{Favorov} generalized amoebas to holomorphic almost periodic functions, 
and Henriques's result was extended to this setting in~\cite{FGS,Silipo}.
Rashkovskii~\cite{Rashkovskii} established a related result, the tube domain $\R^n+i\scrA^c$ of the
complement of an amoeba of a variety of codimension $k$ is $(n{-}k{-}1)$-pseudoconvex, in the sense 
of Rothstein~\cite{Rothstein}, which implies Mikhalkin's result on the absence of $k$-caps.
This work just mentioned, including that of Gromov, has yet to be assimilated by the community
working on amoebas and coamoebas.

Bushueva and Tsikh~\cite{BT12} proved Henriques's conjecture when the variety is a
complete intersection, but the general case remains open.
We first prove a version of Henriques's conjecture for coamoebas.
This involves the closure of the inverse image of the coamoeba in the universal cover $\R^n$ of
$(S^1)^n$, called the lifted coamoeba.
We show that if a subvariety of a torus has codimension $k{+}1$, then the complement of its
lifted coamoeba is $k$-convex.
This uses the phase limit set of the variety $V$~\cite{NS} to reduce the statement to the
$k$-convexity of the complement of an arrangement of affine subspaces of $\R^n$.

Let $\K$ be a complete valued field with residue field $\C$.
The nonarchimedean coamoeba of $V\subset(\K^\times)^n$ is its image under an argument
map~\cite{NSna}.
We show that if $V$ has codimension $k{+}1$, then the complement of the lift of its nonarchimedean
coamoeba is $k$-convex.  

In Section~\ref{S:one} we recall the structure of coamoebas and some work of Henriques needed for
higher convexity.
We use this in Section~\ref{S:two} to prove that the complement of the lifted coamoeba of a variety of
codimension $k{+}1$ is $k$-convex.
We discuss nonarchimedean coamoebas in Section~\ref{S:three} and prove that the complement of the lifted
nonarchimedean coamoeba of a variety $V\subset(\K^\times)^n$ of codimension $k{+}1$ is $k$-convex.

\section{Preliminaries}\label{S:one}
Let \defcolor{$\C^\times$} be the nonzero complex numbers and \defcolor{$(\C^\times)^n$} the
$n$-dimensional complex torus.
Subtori are connected algebraic subgroups of $(\C^\times)^n$.
They have the following classification. 
Let $N\subset\R^n$ be a rational linear subspace ($N$ is spanned by integer vectors).
Then the image of $\C\otimes_\R N$ in $(\C^\times)^n$ under the exponential map is a subtorus
\defcolor{$\C^\times_N$}, and all subtori occur in this way.
Call a coset \defcolor{$a\C^\times_N$} of a subtorus an \demph{affine subgroup} of
$(\C^\times)^n$.

Let $\defcolor{\T}:=\R/\Z\simeq S^1$ be the circle group of complex numbers of norm 1.
Connected subgroups of $\T^n$ have the same classification as subtori of $(\C^\times)^n$,
having the form $\T_N:=N/(N\cap\Z^n)$ for $N$ a rational linear subspace of $\R^n$.
Call  coset $a\T_N$ of a connected
subgroup $\T_N$ (with $a\in\T^n$) an \demph{affine subgroup} of $\T^n$.

We identify $\Z^n$ with the character group of $(\C^\times)^n$, where $\alpha\in\Z^n$
corresponds to the character 
$x\mapsto \defcolor{x^\alpha}:=x_1^{\alpha_1}\dotsb x_n^{\alpha_n}$, which is a 
\demph{Laurent monomial}. 
A finite linear combination 
\[
    f\ =\ \sum a_\alpha x^\alpha
    \qquad a_\alpha\in\C^\times\,,
\]
of Laurent monomials is a \demph{Laurent polynomial}.
The coordinate ring of $(\C^\times)^n$ is the ring of Laurent polynomials, 
$\C[x_1,x_1^{-1},\dotsc,x_n,x_n^{-1}]$, which we will write as \defcolor{$\C[x^{\pm}]$}.

An algebraic subvariety $V\subset(\C^\times)^n$ is the set of common zeroes of a collection
of Laurent polynomials.
We will not assume that varieties are irreducible, but rather that each component has the
same dimension. 
The collection of all Laurent polynomials that vanish on $V$ is an ideal $I=I(V)$ of
$\C[x^{\pm}]$.
As amoebas and coamoebas are set-theoretic objects, we are not
concerned with the scheme structure and will work with varieties.

\subsection{Structure of coamoebas}\label{SS:coamoeba}

Given a vector $w\in\R^n$ and a Laurent polynomial $f$, the \demph{initial form
$\ini_w f$} of $f$ is the sum of terms $a_\alpha x^\alpha$ of $f$ for which 
$w\cdot \alpha := w_1\alpha_1+\dotsb+w_n\alpha_n$ is maximal among all terms of $f$.
Given an ideal $I$ of $\C[x^{\pm}]$ its \demph{initial ideal $\ini_w I$} is the collection
of all initial forms $\ini_w f$ for $f\in I$.
The \demph{initial variety $\ini_w V$} is the variety defined by the initial ideal 
$\ini_w I$, where $I$ is the ideal of $V$.

We have $\R\times\T\xrightarrow{\,\sim\,}\C^\times$ via
$(r,\theta)\mapsto e^{r+2\pi\theta\sqrt{-1}}$.
The inverse map is given by $z\mapsto (\log|z|,\arg(z))$.
This extends to a map $(\C^\times)^n\xrightarrow{\,\sim\,}\R^n\times\T^n$.
Write $\Log\colon(\C^\times)^n\to\R^n$ for the projection to the first factor and 
$\Arg\colon(\C^\times)^n\to\T^n$ for the projection to the second factor.
The \demph{amoeba} $\defcolor{\scrA(V)}\subset\R^n$ of an algebraic
subvariety $V$ is its image under $\Log$, and the \demph{coamoeba}
$\defcolor{\coscrA(V)}\subset\T^n$ of $V$ is its image under $\Arg$.  
The coamoeba of the subgroup $\C^\times_N$ is $\T_N$, and affine subgroups of $\T^n$ are
coamoebas of affine subgroups of $(\C^\times)^n$.

The structure of the boundary of the coamoeba is important for our arguments.
The \demph{phase limit set $\scrP^{\infty}(V)$} of an algebraic subvariety
$V\subset(\C^\times)^n$ is the set of accumulation points of arguments of unbounded sequences
in $V$.
We state the main result of~\cite{NS}.

\begin{proposition}
 The closure of the coamoeba $\coscrA(V)$ of an algebraic subvariety $V$ of $(\C^\times)^n$ is
 $\coscrA(V)\cup\scrP^{\infty}(V)$, and  
\[
   \scrP^{\infty}(V)\ =\ 
    \bigcup_{w\in\R^n\smallsetminus\{0\}} \coscrA( \ini_w V)\,.
\]
\end{proposition}

To discuss the structure of $\scrP^\infty(V)$, we first explain the relation between the
tropical variety of $V$ and its initial varieties. 
The tropical variety~\cite{Berg,BiGr,SS} $\scrT(V)$ of $V$ is the set of those $w\in\R^n$
such that $\ini_w V\neq\emptyset$.
This set can be given (noncanonically) the structure of a rational polyhedral fan
$\Sigma$, called a \demph{tropical fan}, with the following properties.
If $w,v$ lie in the relative interior of a cone $\sigma$ of $\Sigma$, then 
$\ini_w V=\ini_v V$.
Write \defcolor{$\ini_\sigma V$} for this common initial variety.
Furthermore, if $\sigma\subset\tau$, then $\ini_\tau V$ is an initial variety of
$\ini_\sigma V$, and all initial varieties of $\ini_\sigma V$ occur in this way.
For $\sigma\in\Sigma$, the initial variety $\ini_\sigma V$ is equivariant for the torus
\defcolor{$\C^\times_{\langle\sigma\rangle}$}, where 
\defcolor{$\langle\sigma\rangle$} is the linear span of $\sigma$ in $\R^n$, and every
initial variety of $\ini_\sigma V$ is also $\C^\times_{\langle\sigma\rangle}$-equivariant.
The quotient $W$ of $\ini_\sigma V$ by the torus  $\C^\times_{\langle\sigma\rangle}$ is a
subvariety of $(\C^\times)^n/\C^\times_{\langle\sigma\rangle}$, and the initial varieties of $W$ are
exactly the quotients of the initial varieties of $\ini_\sigma V$ by the torus
$\C^\times_{\langle\sigma\rangle}$. 

This relation between cones of $\Sigma$ and initial varieties has a partial converse.
If $\ini_\sigma V$ is equivariant for a subgroup $\C^\times_N$ that properly contains
$\C^\times_{\langle\sigma\rangle}$, then there is a cone $\tau$ of $\Sigma$ with
$\sigma\subsetneq\tau$ with $\ini_\tau V=\ini_\sigma V$.
Lastly, if $V$ has pure dimension $d$, then the fan $\Sigma$ is a pure polyhedral complex
of dimension $d$.
It follows that if $\sigma$ is a maximal cone of $\Sigma$, then $\ini_\sigma V$ is
supported on a finite union of orbits (and hence cosets) of
$\C^\times_{\langle\sigma\rangle}$.

This has the following consequence for the closure of the coamoeba $\overline{\coscrA(V)}$
and the phase limit set $\scrP^\infty(V)$ when $V$ has pure dimension $d$, as shown in~\cite{NS}.
Let $\Sigma$ be a tropical fan for $V$.
Then $\scrP^\infty(V)$ is a finite union
\[
   \scrP^\infty(V)\ =\ 
    \bigcup_{\sigma\in\Sigma\smallsetminus\{0\}} \coscrA(\ini_\sigma V)\,,
\]
where each initial coamoeba $\coscrA(\ini_\sigma V)$ is equivariant for the
subtorus $\T_{\langle\sigma\rangle}$, and if $\sigma$ is a maximal cone of $\Sigma$, then
$\coscrA(\ini_\sigma V)$ is a finite union of cosets of $\T_{\langle\sigma\rangle}$.
The \demph{shell} of $V$ is the union of such cosets
\[
    \defcolor{{\it sh}(V)}\ :=\ \bigcup_{\sigma\in\Sigma\mbox{{\scriptsize\rm \ is maximal}}} \coscrA(\ini_\sigma V)\,.
\]  
It is an arrangement of finitely many affine subgroups of dimension $d=\dim V$.

The closure of the pullback of a coamoeba $\coscrA(V)\subset\T^n$ to the universal cover
$\R^n$ of $\T^n$ is the \demph{lifted coamoeba \defcolor{$\lcoscrA(V)$}} of $V$.
The pullback of the shell is the lifted shell \defcolor{$\lsh(V)$} of $V$, which is a
$\Z^n$-periodic arrangement of translates of rational linear subspaces of $\R^n$, each of
dimension $d=\dim V$.
While the lifted shell consists of infinitely many linear spaces, at most finitely many
will meet any compact subset, as only finitely many meet any fundamental domain of
$\Z^n$ in $\R^n$.

\begin{remark}\label{R:covers}
 For $m$ a positive integer, consider the map 
 $\varphi_m\colon(\C^\times)^n\to(\C^\times)^n$ 
 given by $(z_1,\dotsc,z_n)\mapsto(z_1^m,\dotsc,z_n^m)$, which is a $m^n$-fold cover.
 This induces a $m^n$-fold cover $\T^n\to\T^n$.
 The pullback of $V$ along $\varphi_m$ is a subvariety  of
 $(\C^\times)^n$ whose coamoeba is the pullback of the coamoeba of $V$ along $\varphi_m$.
\end{remark}

We illustrate these ideas for lines in $(\C^\times)^3$.
On the left below is the coamoeba of a line defined over $\R$ and on the right is one that 
is not.
Both are displayed in a fundamental domain for $\T^3$ in $\R^3$.
\[
  \includegraphics[height=100pt]{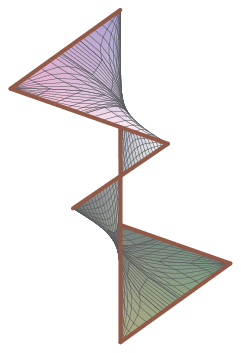}
   \qquad
  \includegraphics[height=100pt]{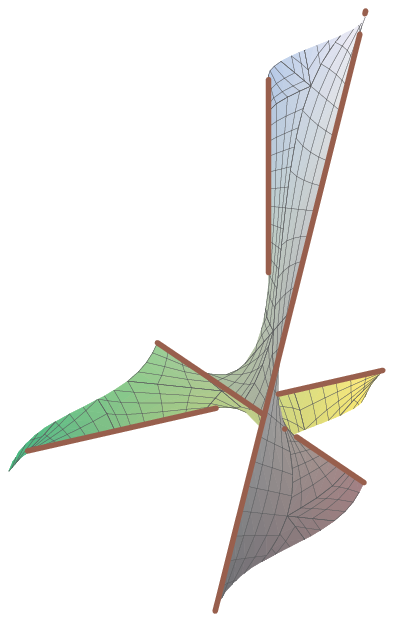}
\]
The lines bounding these coamoebas are their phase limit sets, and also their shells.

Figure~\ref{F:one} shows part of their lifted coamoebas, displaying them in eight
($2^3$) fundamental domains.
\begin{figure}[htb]
  \includegraphics[height=190pt]{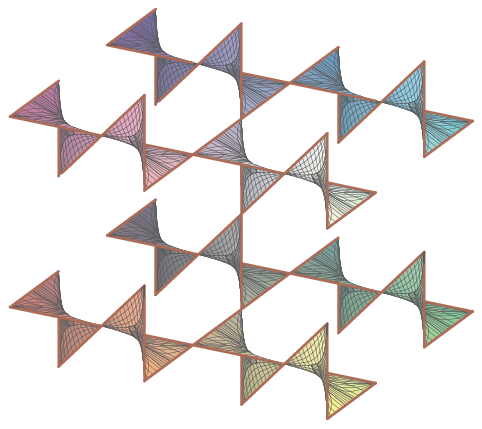}
   \qquad
  \includegraphics[height=190pt]{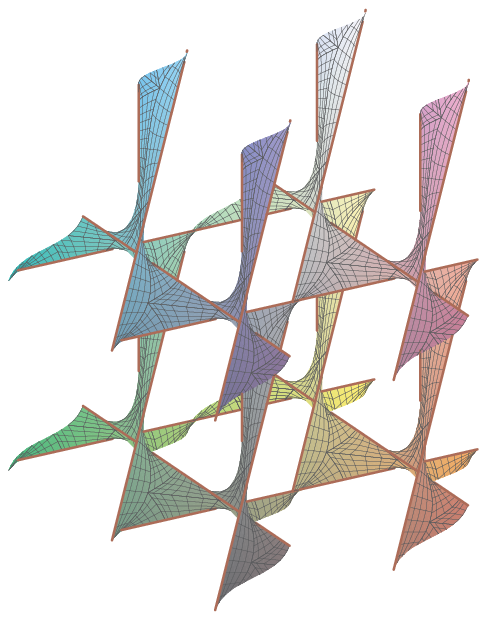}
\caption{Lifted coamoebas.}\label{F:one}
\end{figure}
The periodic arrangements of lines are their lifted shells.
These are also the coamoebas of the pullbacks under $\varphi_2^{-1}$ in a single
fundamental domain.  

\subsection{Higher convexity}
A subset  $X$ of $\mathbb{R}^n$ is \demph{convex} if for any affine  line $L$ in
$\mathbb{R}^n$, the intersection $L\cap X$ has at most one connected component. 
In other words, the intersection  $L\cap X$ contains all intervals with boundary in
$L\cap X$.
If the  points of $X$ are viewed as  $0$-cycles, then the convexity of $X$  means that  if
$a$ and $b$  are $0$-cycles that are homologous in $X$, then they are also homologous
in $L\cap X$ where $L$ is the line containing  $a$ and $b$.  

Write \defcolor{$\witi{H}_{*}(X)$} for the reduced homology of a
space $X$ with integral coefficients.
This is the kernel of the map $\deg\colon H_*(X)\to H_*({\rm pt})$ induced by the
map $X\to{\rm pt}$ to a point.
Gromov~\cite[\S$\tfrac{1}{2}$]{Gromov} gave the following generalization of convexity.
A subset  $X$ of a vector space $V$ is \demph{$k$-convex} if for any affine $(k{+}1)$-plane
$L$, the map $\witi{H}_k(L\cap X)\rightarrow \witi{H}_k (X)$
induced by the inclusion is injective.  
Connected and $0$-convex is the usual notion of convexity.
For us, $X$ will always be a triangulable open subset of $\R^n$.

Henriques later considered this same notion, applying it to complements
of amoebas~\cite{He04}.
We need some technical results of Henriques, which may be
found in Section~2 of~\cite{He04}.
We will always take homology with integer coefficients.
Let $X$ be a triangulated manifold, such as an open subset of an affine space.
For a nonnegative integer $k$, let $\Delta_k$ be the standard $k$-simplex.
Let \defcolor{$\Pl_k(X)$} be the subgroup of the group of singular $k$-chains in $X$
generated by piecewise linear maps $\lambda\colon\Delta_k\to X$, called simplices.
Two chains $\lambda$ and $\mu$ are \demph{geometrically equivalent} if they cannot be
distinguished by $k$-forms.
We define \defcolor{$C^\Delta_\bullet(X)$} to be the quotient of $\Pl_\bullet(X)$ by
geometric equivalence and write $[\lambda]\in C^\Delta_\bullet(X)$ for the class of 
a chain $\lambda\in\Pl_\bullet(X)$.
These \demph{polyhedral chains} were introduced by Whitney~\cite{Whitney}.
They form a subcomplex of the complex of singular chains on $X$, and the first observation
of Henriques~\cite[Prop.~2.3]{He04} is that the homology of this complex is identified
with the usual singular homology of $X$, and the same for reduced homology.
We will always assume that homology classes are represented by polyhedral
chains.

A chain $\lambda\in\Pl_k(X)$ is a sum,  $\lambda=\sum_i a_i\lambda_i$, where for each $i$,
$a_i\in\Z$ and $\lambda_i\colon\Delta_k\to X$ is a piecewise linear map.
The \demph{support, $\supp(\lambda)$} of such a chain is the union of
the images of the maps $\lambda_i$.
If $c=[\lambda]\in C^\Delta_k(X)$ is a polyhedral chain, its support
is 
\[
   \supp(c)\ :=\ \bigcap_{c=[\lambda]} \supp(\lambda)\,.
\]
Henriques~\cite[Lemma~2.4]{He04} shows that if $L$ is an affine space of dimension $k$,
then any polyhedral $k$-chain $c\in C^\Delta(L)$ has a representative
$\lambda\in\Pl_k(L)$ with $\supp(c)=\supp(\lambda)$.
This has the corollary that if $X\subset L$ is triangulated, then 
$C^\Delta_k(X)$ can be identified with the chains on $L$ whose support is contained in
$X$. 

We need a statement that Henriques establishes in the proof of his Lemma~2.4.

\begin{lemma}\label{L:pm_decomp}
 Let $L$ be an oriented affine linear space of dimension $k$ and 
 $c\in C^\Delta_k(L)$ be a polyhedral chain.
 Then $c$ has a representative $\lambda\in\Pl_k(L)$ with $\supp(c)=\supp(\lambda)$,
\[
   \lambda\ =\ \sum_{i} a_i\lambda_i\,,
\]
 where, for each $i$, $\lambda_i\colon\Delta_k\to L$ is an orientation-preserving affine
 map, and if $i\neq j$ the intersection $\supp(\lambda_i)\cap\supp(\lambda_j)$ is
 a common proper face of the images of $\lambda_i$ and $\lambda_j$.
\end{lemma}

Let $\lambda$ be the cycle of this lemma.
If we define 
$\defcolor{\lambda_{\pm}}:=\sum_i \max\{0,\pm a_i\}\lambda_i$, then 
$\lambda=\lambda_+ - \lambda_-$ and the supports of $\lambda_+$ and $\lambda_-$ intersect in
a set of dimension at most $k{-}1$.
By definition, each of $\lambda_\pm$ is a positive sum of positively oriented simplices.
We deduce a useful consequence of this decomposition.

\begin{corollary}\label{C:disjoint_supports}
 Let $c$ and $\lambda=\lambda_+-\lambda_-$ be as in Lemma~$\ref{L:pm_decomp}$.
 Then 
 \begin{equation}\label{Eq:disjoint_supports}
   \supp(\partial\lambda)\ =\ 
   \supp(\partial\lambda_+) \bigcup \supp(\partial\lambda_-)\,.
 \end{equation}
\end{corollary}

\begin{proof}
 If~\eqref{Eq:disjoint_supports} does not hold, then there is a $k{-}1$ simplex
 $\mu\colon\Delta_{k-1}\to L$ occurring
 in both $\partial\lambda_+$ and $\partial\lambda_-$, but not in 
 $\partial\lambda=\partial\lambda_+-\partial\lambda_-$. 
 Necessarily, the coefficients of $\mu$ in $\partial\lambda_+$ and in $\partial\lambda_-$
 are equal and nonzero, for they cancel in 
 $\partial\lambda=\partial\lambda_+ -\partial\lambda_-$.

 Thus there are $k$-simplices $\lambda_i$ in $\lambda_+$ and $\lambda_j$ in $\lambda_-$ for
 which $\mu$ appears in their boundary with the same sign.
 However, as $\lambda_i$ and $\lambda_j$ both preserve orientation, this implies that they
 lie on the same side of the image of $\mu$, contradicting the assertion of
 Lemma~\ref{L:pm_decomp} that their supports intersect in common proper face.
\end{proof}

Let $\witi{C}^\Delta_\bullet(X)$ be the complex of reduced chains (the kernel of the map
$C^\Delta(X)\to \Z$ induced by the degree of a 0-chain).
Let $\witi{Z}_k(X)$ be the group of reduced $k$-cycles of $X$, which is the kernel
of the map $\partial_k\colon \witi{C}^\Delta_k(X)\to\witi{C}^\Delta_{k-1}(X)$.
We state Lemma 2.7 of~\cite{He04}.

\begin{lemma}\label{L:2.7}
  Let $k\geq 0$ and let  $X\subset L$ be an open subset of a $k$-dimensional real vector
  space $L$.
  Let $c\in\witi{Z}_{k}(X)$.
  Then there is a unique polyhedral chain $C\in C^\Delta_{k+1}(L)$ with $\partial C=c$.
  Moreover $[c]$ is nonzero in $\witi{H}_{k}(X)$ if and only if there is a point
  $p\in L\smallsetminus X$ in the support of $C$.
  In particular, every representative $\lambda\in\Pl_{k+1}(L)$ of $C$ meets 
  $L\smallsetminus X$.
\end{lemma}

\section{Higher convexity of coamoeba complements}\label{S:two}

We develop some properties of higher convexity and use them to show that the
complement of a lifted coamoeba of a variety of codimension $k{+}1$ is $k$-convex.
Gromov~\cite[\S$\tfrac{1}{2}$]{Gromov} showed that the intersection of two $k$-convex
subsets of $\R^n$ is again $k$-convex. 

\begin{proposition}\label{P:intersection}
  If $X,Y\subset\R^n$ are $k$-convex, then so is $X\cap Y$.
\end{proposition}

\begin{proof}
 Let $L$ be an affine $(k{+}1)$-plane in $\R^n$.
 Consider the Mayer-Vietoris sequences for $X\cup Y$ and $L\cap(X\cup Y)$, 
\[
 \xymatrix{
  \witi{H}_{k+1}(L\cap(X\cup Y)) \ar[d]^{\iota_{X\cup Y}} \ar[r]^(.52){\partial}&
  \witi{H}_{k}(L\cap(X\cap Y)) \ar[d]^{\iota_{X\cap Y}} \ar[r]^(.43){i} &
  \witi{H}_{k}(L\cap X)\oplus\witi{H}_k(L\cap Y)
     \ar[d]_{\iota_X\oplus\iota_Y}\\
  \witi{H}_{k+1}(X\cup Y) \ar[r]^(.52){\partial}&
  \witi{H}_{k}(X\cap Y)  \ar@{->}[r]^(.43){i}&
  \witi{H}_{k}(X)\oplus\witi{H}_k(Y).
 }
\]
 Since $L$ is an affine space of dimension $k{+}1$, 
 $\witi{H}_{k+1}(L\cap(X\cup Y))=0$, so the map $i$ in the top row is injective.
 As $X$ and $Y$ are $k$-convex, the map $\iota_X\oplus\iota_Y$ is injective.
 It follows from the right commutative square that the map $\iota_{X\cap Y}$ is also
 injective. 
\end{proof}

An important class of $k$-convex spaces are complements of collections of 
codimension $k{+}1$ affine linear subspaces.
A collection $S$ of affine subspaces of $\R^n$ is \demph{locally finite} if at most
finitely many members of $S$ meet any given compact subset of $\R^n$.

\begin{lemma}\label{L:affine_complement}
 Let $S$ be a locally finite collection of affine linear subspaces of\/ $\R^n$ of
 codimension $k{+}1$.
 Then its complement $\R^n\smallsetminus S$ is $k$-convex.
\end{lemma}

\begin{proof}
 Let $L$ be an affine linear space of dimension $k{+}1$.
 We show that the map 
\[
   \iota\ \colon\ \witi{H}_{k}(L\smallsetminus S)
    \ \longrightarrow\ 
   \witi{H}_{k}(\R^n\smallsetminus S)
\]
 induced by the inclusion $L\smallsetminus S\hookrightarrow\R^n\smallsetminus S$ is
 injective. 
 First, suppose that $S$ consists of a single subspace $M$ of codimension $k{+}1$.
 If $L\cap M=\emptyset$ or if $\dim L\cap M\geq 1$, then 
 $\witi{H}_{k}(L\smallsetminus M)=0$, so there is nothing to show.
 Otherwise, $L\cap M=\defcolor{p}$, a point.
 Since we have the isomorphism 
 $\R^n\smallsetminus M\simeq (L\smallsetminus p)\times M$ and $M$ is contractible, the
 inclusion $L\smallsetminus S\to\R^n\smallsetminus S$ is a homotopy equivalence, which
 implies that $\iota$ is an isomorphism.
 Proposition~\ref{P:intersection} implies the result by induction if $S$ is finite.

 Now suppose that $S$ is not finite.
 Let $c\in Z_{k}(L\smallsetminus S)$ be a $k$-cycle whose image 
 $[c]\in \witi{H}_{k}(L\smallsetminus S)$ is nonzero.
 By Lemma~\ref{L:2.7}, there is a unique polyhedral chain $C$ in $C^\Delta_{k+1}(L)$ with  
 $\partial C=c$, and the support of $C$ meets $S$.
 Let $\lambda\in\Pl_{k+1}(X)$ be any chain representing $C$, then its support meets $S$.
 Let $S_{{\rm finite}}$ be the (finite) union of affine linear spaces from $S$ that meet
 the support of $\lambda$.
 By Lemma~\ref{L:2.7} again, 
 $[c]\in\witi{H}_{k}(L\smallsetminus S_{{\rm finite}})$ is nonzero, and so
 $\iota[c]\in\witi{H}_{k}(\R^n\smallsetminus S_{{\rm finite}})$ is nonzero, as
 $S_{{\rm finite}}$ is a finite collection of affine subspaces of codimension $k{+}1$.
 By the commutative diagram
\[
 \xymatrix{
   \witi{H}_{k}(L\smallsetminus S) \ar[d]_{\iota} \ar[r]^(.43){i_*}   &
   \witi{H}_{k}(L\smallsetminus S_{{\rm finite}}) \ar[d]_{\iota_{{\rm finite}}} \\
   \witi{H}_{k}(\R^n\smallsetminus S) \ar[r]^{i_*}       &
   \witi{H}_{k}(\R^n\smallsetminus S_{{\rm finite}}) 
 }
\]
 we have that $0\neq \iota_{{\rm finite}} i_*[c]= i_*\iota[c]$, 
 so  $\iota[c]\in\witi{H}_{k}(\R^n\smallsetminus S)$ is also nonzero.
\end{proof}

The lifted shell of a subvariety of the torus $(\C^\times)^n$ of pure codimension $k{+}1$ 
is a locally finite collection of affine subspaces of $\R^n$ of codimension $k{+}1$.
We deduce a corollary.

\begin{corollary}
 The complement of the lifted shell of a subvariety of $(\C^\times)^n$ of pure codimension
 $k{+}1$ is $k$-convex.
\end{corollary}

In Subsection~\ref{SS:proof} we prove the following topologial result relating 
the complement of a coamoeba to the complement of its shell.
Let $\pi$ be an affine subgroup of $\T^n$ of dimension $k{+}1$.
Note that $\pi$ is orientable, has an affine structure induced
from its universal cover, and may be triangulated. 
A chain $\lambda\in\Pl_{k+1}(\pi)$ is \demph{positive} if there is an orientation of $\pi$
for which we may write $\lambda=\sum_i a_i\lambda_i$ where each $a_i>0$ and each simplex
$\lambda_i\colon\Delta_{k+1}\to\pi$ is affine linear and preserves the orientation.

\begin{lemma}\label{L:crucial}
 Let $V\subset(\C^\times)^n$ be a subvariety of codimension $k{+}1$ and 
 let $\pi\subset\T^n$ be an affine subgroup of dimension $k{+}1$ which is the coamoeba of
 an affine subgroup $a\C^\times_N$ that meets $V$ transversally.
 Let $\gamma\in Z_{k}(\pi\smallsetminus\overline{\coscrA(V)})$
 be a nonzero cycle which bounds a positive chain $\lambda\in \Pl_{k+1}(\pi)$, 
 $\partial\lambda=\gamma$.
 If the support of $\lambda$ meets $\overline{\coscrA(V)}$, then it meets the shell of
 $V$. 
\end{lemma}

Let $\lcoscrA(V)$ and $\lsh(V)$ be the lifted coamoeba and lifted shell of $V$, which are
closed subsets of $\R^n$.
The key step in showing that $\R^n\smallsetminus \lcoscrA(V)$ is $k$-convex is to
reduce it to $\R^n\smallsetminus\lsh(V)$.
An affine subspace $L\subset\R^n$ is \demph{rational} if it is the translate of a rational
linear subspace of $\R^n$.
A rational affine subspace $L\subset\R^n$ of dimension $k{+}1$ is \demph{general for $V$}
if there is an affine subgroup $a\C^\times_N\subset(\C^\times)^n$ which meets $V$
transversally and $L$ is the lifted coamoeba of $a\C^\times_N$.
Kleiman's Transversality Theorem~\cite{KL74} implies that for any affine rational subspace
$L\subset\R^n$ of dimension $k{+}1$, its translates which are general for $V$ form a dense
subset of all translates.

\begin{lemma}\label{L:reductionToShell}
 For any affine rational $(k{+}1)$-plane $L$ of\/ $\R^n$ that is general for $V$, the map
 on reduced homology induced by the inclusion
\[
   i_*\ \colon\ \witi{H}_{k}(L\smallsetminus\lcoscrA(V))\ 
    \longrightarrow\ \witi{H}_{k}(L\smallsetminus\lsh(V))
\]
 is an injection.
\end{lemma}

\begin{proof}
 Let $[c]\in\witi{H}_{k}(L\smallsetminus\lcoscrA(V))$ be nonzero.
 By Lemma~\ref{L:2.7}, there is a polyhedral chain $C\in C^\Delta_{k+1}(L)$ with 
 $\partial C=c$ and $\supp(C)\cap\lcoscrA(V)\neq\emptyset$.
 Let $\lambda\in \Pl_{k+1}(L)$ represent $C$ with $\supp(C)=\supp(\lambda)$.

 Since $\supp(\lambda)$ meets $\lcoscrA(V)$, if we write
 $\lambda=\lambda_+-\lambda_-$ as we did following Lemma~\ref{L:pm_decomp}, then 
 one of $\lambda_\pm$ has support meeting $\lcoscrA(V)$.
 Suppose that it is $\lambda_+$.
 By Corollary~\ref{C:disjoint_supports}, we have 
 $\gamma=\gamma_+-\gamma_-$, where $\gamma_\pm:=\partial\lambda_\pm$, and
 $\gamma_\pm\in \Pl_{k}(L\smallsetminus\lcoscrA(V))$.
 Under the cannonical projection $\varphi\colon\R^n\to\R^n/\Z^n=\T^n$ the image
 $\pi=\varphi(L)$ of $L$ is an affine subgroup of $\T^n$.

 To apply Lemma~\ref{L:crucial}, we need to ensure that
 $\varphi_*(\gamma_+)$ is nonzero.
 For this, we may replace the map $\varphi\colon\R^n\to\T^n=\R^n/\Z^n$
 by a finite cover $\R^n\to\R^n/m\Z^n$, where $m$ is large enough so that the supports of
 $\gamma_+$ and $\lambda_+$ lie in the interior of a fundamental chamber of the lattice
 $m\Z^n$.
 Then we need to replace $V$ and its coamoeba $\coscrA(V)$ by their inverse images as in
 Remark~\ref{R:covers}.

 Thus we may assume that $\varphi_*\gamma_+\in Z_{k-1}(\pi\smallsetminus\overline{\coscrA(V)})$ is a
 nonzero cycle which bounds the positive cycle $\varphi_*\lambda_+$, and the support of
 $\varphi_*\lambda_+$ meets $\coscrA(V)$. 
 By Lemma~\ref{L:crucial}, the support of $\varphi_*\lambda_+$ meets $\lsh(V)$.
 It follows that the support of $\lambda_+$ meets $\lsh(V)$ and therefore the support of
 $\lambda$  meets $\lsh(V)$.
 Applying Lemma~\ref{L:2.7} again, we deduce that $i_*[c]$ is nonzero in 
 $\witi{H}_{k}(L\smallsetminus\lsh(V))$.
\end{proof}

We now deduce the main result of this section.

\begin{theorem}
 Let $V\subset(\C^\times)^n$ be an algebraic variety of codimension $k{+}1$.
 Then the complement $\R^n\smallsetminus\lcoscrA(V)$ of its lifted coamoeba is
 $k$-convex. 
\end{theorem}

\begin{proof}
 We will show that 
 $\iota\colon \witi{H}_{k}(L\smallsetminus\lcoscrA(V))\to
   \witi{H}_{k}(\R^n\smallsetminus\lcoscrA(V))$ is an injection, for all 
 $(k{+}1)$-dimensional affine subspaces $L$ of $\R^n$.
 The arguments of Lemma~3.6 of~\cite{He04} (for nonnegative classes) hold
 more generally for all classes, and imply that it suffices to show this when 
 $L\subset\R^n$ is a rational affine subspace.
 Moreover, we may replace the subspace $L$ by any sufficiently nearby translate and therefore we
 may assume that $L$ is general for $V$.

 Consider the commutative square
\[
 \xymatrix{
   \witi{H}_{k}(L\smallsetminus\lcoscrA(V)) \ar[d]_{\iota} \ar[r]^{i_*}   &
   \witi{H}_{k}(L\smallsetminus\lsh(V)) \ar[d]_{\iota_{\widehat{{\it sh}}(V)}} \\
   \witi{H}_{k}(\R^n\smallsetminus\lcoscrA(V)) \ar[r]^{i_*}       &
   \witi{H}_{k}(\R^n\smallsetminus\lsh(V)).
 }\ 
\]
 By Lemma~\ref{L:reductionToShell} the map $i_*$ in the first row is injective,
 and by Lemma~\ref {L:affine_complement} the map $\iota_{\widehat{{\it sh}}(V)}$ is injective.
 It follows that the map $\iota$ is injective.
\end{proof}

\subsection{Proof of Lemma~\ref{L:crucial}}\label{SS:proof}
We first consider the case when $V$ is a hypersurface, $k=0$, so that
$\pi\simeq S^1$ is one-dimensional.
This recovers Nisse's result~\cite{Ni09}.

\begin{lemma}\label{Lem:hypersurface}
 Let $V\subset(\C^\times)^n$ be a hypersurface.
 Let $\pi$ be a one-dimensional affine subgroup of \/$\T^n$ which is the coamoeba of a
 one-dimensional affine subgroup of $(\C^\times)^n$ that meets $V$ transversally.
 If $p,q\in\pi$ do not lie in the closure of $\coscrA(V)$, but one
 of the arcs $\arc{pq}$ of $\pi$ they subtend meets $\coscrA(V)$, then that
 arc meets the phase limit set of $V$.
\end{lemma}

\begin{proof}
 Applying a translation if necessary, we may assume that $\pi$ is a one-dimensional subgroup
 $\T_N$ of $\T^n$ and that $V$ meets the corresponding subgroup $\C^\times_N$ of
 $(\C^\times)^n$ transversally in (at least) a point with argument in the interior of the
 arc $\arc{pq}$.  
 Let $\defcolor{S_{pq}}\subset \C^\times_N$ be its sector of points with argument in the arc
 $\arc{pq}$. 
 Identify $\Arg^{-1}(\arc{pq})$ with the product $\R^{n-1}_>\times S_{pq}$, where
 \defcolor{$\R_>$} is the set of strictly positive real numbers and $\R^{n-1}_>$ is $(\R_>)^{n-1}$.

 Suppose by way of contradiction that the phase limit set
 $\scrP^\infty(V)$ does not meet the arc $\arc{pq}$.
 As $\scrP^\infty(V)$ consists of accumulation points of unbounded sequences in $V$, we
 have that $\Arg^{-1}(\arc{pq}^\circ)\cap V$ is compact, where $\arc{pq}^\circ$
 is the relatively open arc, $\arc{pq}\smallsetminus\{p,q\}$.
 Otherwise, there is an unbounded sequence in $V$ whose arguments lie in $\arc{pq}^\circ$.
 This sequence has a subsequence whose arguments converge to a point in $\arc{pq}$, which
 is necessarily in $\scrP^\infty(V)$, contradicting our assumption that $\scrP^\infty(V)$
 does not meet $\arc{pq}$.

 Writing $S^\circ_{pq}$ for the open sector, we have 
 $\Arg^{-1}(\arc{pq}^\circ)=\R^{n-1}_>\times S^\circ_{pq}$, and thus the projections
 \defcolor{$K$} and \defcolor{$K_{pq}$} of $\Arg^{-1}(\arc{pq}^\circ)\cap V$ to
 $\R^{n-1}_>$ and $S^\circ_{pq}$ (respectively) are compact.
 Note that $1\in K$ as $\emptyset\neq S_{pq}\cap V\subset \C^\times_N\cap V$.
 Figure~\ref{F:sectors} displays a schematic diagram of the 
\begin{figure}[htb]
  \begin{picture}(102,102)
    \put(0,0){\includegraphics{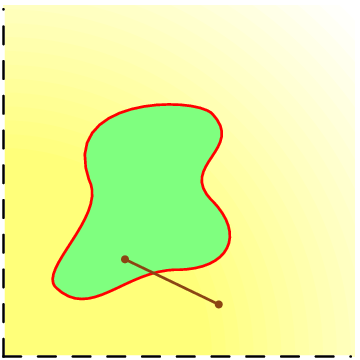}}
    \put(38,50){$K$}
    \put(29,26){$1$} \put(66,13){$a$}
    \put(68,82){$\R^{n-1}_>$}
  \end{picture}
   \qquad\qquad
  \begin{picture}(158,102)
   \put(19,0){\includegraphics{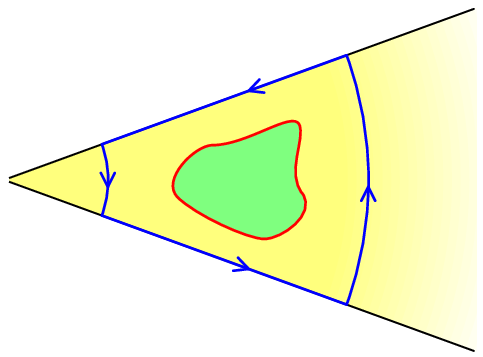}}
   \put(145,88){$q$} \put(135,47){$S_{pq}$}
   \put(145,12){$p$}
   \put(82,47){$K_{pq}$}
   \put(75,77){$\gamma$}
   \put(15,80){$\C^\times_N$}
  \end{picture}
 \caption{Projections of $\Arg^{-1}(\protect\arc{pq}^\circ)\cap V$.}
 \label{F:sectors}
\end{figure}
 two sets $K\subset\R^{n-1}_>$ and  $K_{pq}\subset \C^\times_N$.
 
Let $\defcolor{\gamma}\subset S_{pq}$ be a 1-cycle enclosing $K_{pq}$ that is positively
oriented with respect to the orientation of $\C^\times_N$, as shown in Figure~\ref{F:sectors}. 
Then no translate $a\gamma$ for $a\in\R^{n-1}_>$ can meet $V$.
We construct a 2-cycle $c\subset (\C^\times)^n$ that is closed, $\partial c=0$, such that
the topological intersection number $[c]\cdot V$ of the homology class $[c]$ with $V$ is
nonzero. 
This will be a contradiction for $c$ will be contractible in $(\C^\times)^n$ and thus
$[c]=0$. 
Technically, we need this intersection to take place in a compact manifold.
To that end, let $\overline{V}\subset\P^n$ be its closure.
Then, in $H_*(\P^n,\Z)$, we have $[c]\cdot\overline{V}\neq 0$, but as $c$ is
contractible its class in homology is zero, so $[c]\cdot \overline{V}=0$, which is a
contradiction. 
As $c\subset(\C^\times)^n$, the intersection number $[c]\cdot\overline{V}$ may be computed
using the cycle $c$ and subvariety $V$ of the torus.
See~\cite[pp.~49--53]{GH94} for a discussion of intersection numbers in homology.

Since $\gamma$ does not meet $V$, the intersection $\C^\times_N\cap V$ of algebraic subvarieties
of $(\C^\times)^n$ consists of finitely many points, which we illustrate below.
\[
    \begin{picture}(182,102)(-27,0)
    \put(-27,60){$\C^\times_N$}
    \put(-17,0){\includegraphics{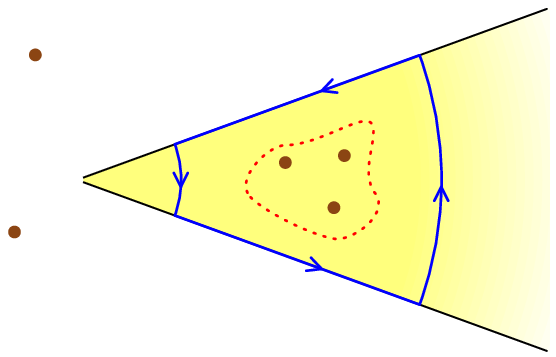}}
    \put(145,88){$q$} \put(135,47){$S_{pq}$}
    \put(145,12){$p$}
    \put(75,77){$\gamma$}
    \put(12, 9){$V\cap\C^\times_N$}
    \put(17,20){\line(-1,1){13.7}}
    \put(24,20){\line(-1,4){16}}
    \put(31,20){\line(4,3){45.3}}
    \put(39,20){\line(3,2){54.5}}
    \put(47,20){\line(2,1){43.5}}
   \end{picture}
\]
Let $\defcolor{c_1}\subset\C^\times_N$ be the 2-chain enclosed by $\gamma$,
with orientation inherited from $\C^\times_N$.
Then, as $\C^\times_N$ meets $V$ transversally, the intersection number of $c_1$ with $V$
is simply the number of points in $c_1\cap V$, which is positive as 
$c_1\cap V = S_{pq}\cap V\neq\emptyset$.

Let $\defcolor{a}\in\R^{n-1}_>\smallsetminus K$ so that $a\C^\times_N\cap V=\emptyset$.
Let \defcolor{$[1,a]$} be the straight line path in $\R^{n-1}_>$ between $1$ and $a$, and
set $\defcolor{c_2}:=[1,a]\gamma\simeq [1,a]\times\gamma$, oriented so that $\gamma$ with its
opposite orientation is one component of its boundary (the other is $a\gamma$).
Finally, let \defcolor{$c_3$} be $ac_1$, but with the opposite orientation to $c_1$ and set 
$\defcolor{c}:=c_1+c_2+c_3$.
Then 
\[
   \partial c \ =\ \gamma\, -\, \gamma + a\gamma\, -\,a\gamma\ =\ 0\,, 
\]
so $c$ is a cycle.
Furthermore, we have
\[
   c\cdot V\ =\ c_1\cdot V\, +\, c_2\cdot V\, +\, c_3\cdot V\ =\ c_1\cdot V\ >\ 0\,,
\]
 as $c_2$ and $c_3$ are disjoint from $V$.
 Lastly, as $c\subset \R^{n-1}_>\times S_{pq}$, it is contractible.
\end{proof}

\begin{corollary}\label{C:shell}
  With the above definitions, $\arc{pq}^\circ$ meets the shell of $V$.
\end{corollary}

\begin{proof}
 Suppose this is not the case.
 Let $\Sigma$ be a tropical fan for $V$, and let $\sigma\in\Sigma$ be a cone of maximal
 dimension such that $\coscrA(\ini_\sigma V)\cap \arc{pq}\neq\emptyset$.
 By Lemma~\ref{Lem:hypersurface}, $\sigma\neq 0$ and so by the maximality of $\sigma$ we
 have that $\C^\times_{\langle\sigma\rangle}$ is the largest subtorus of $(\C^\times)^n$ that acts
 on $\ini_\sigma V$. 
 Also, $\sigma$ is not a maximal cone of $\Sigma$ as $\arc{pq}$ does not meet the shell of
 $V$, and the shell of $V$ consists of the initial coamoebas $\coscrA(\ini_\tau V)$ for
 $\tau$ a maximal cone in $\Sigma$.  

 Replace $V$ by $\ini_\sigma V$.
 Then $\C^\times_{\langle \sigma\rangle}$ acts on $V$ and its coamoeba 
 $\T_{\langle \sigma\rangle}$ acts on $\coscrA(V)$.
 Consider $W:=V/\C^\times_{\langle \sigma\rangle}\subset (\C^\times)^n/\C^\times_{\langle \sigma\rangle}$
 as well as its coamoeba $\coscrA(W)=\coscrA(V)/\T_{\langle \sigma\rangle}$ in  
 $\T'=\T^n/\T_{\langle \sigma\rangle}$.
 As in the proof of Lemma~\ref{Lem:hypersurface}, we may assume that there is a
 one-dimensional subgroup $\T_N$ of $\T^n$ with $\pi=\T_N$ and so $p,q\in\T_N$.
 Since $p,q\not\in\coscrA(V)$, we have that $V\cap \C^\times_N$ is finite.
 
 Let $\pi'$, $p'$, and $q'$ be the images of $\pi$, $p$, and $q$ in $\T'$, respectively.
 Since $V\cap\C^\times_N$ is finite and $V$ is $\C^\times_{\langle\sigma\rangle}$-equivariant, the
 image of $\C^\times_N$ in the quotient $(\C^\times)^n/\C^\times_{\langle\sigma\rangle}$
 is nontrivial, so that $\pi'$ is a one-dimensional subgroup of $\T'$.
 Then the image of the arc $\arc{pq}$ is contained in an arc
 $\barc{p'q'}$ of $\pi'$ which meets $\coscrA(W)$.
 By Lemma~\ref{Lem:hypersurface}, this arc then meets $\scrP^{\infty}(W)$, and thus the
 coamoeba of some initial variety of $W$.

 The inverse image of this initial variety under the map
 $(\C^\times)^n\twoheadrightarrow(\C^\times)^n/\C^\times_{\langle\sigma\rangle}$  is an initial variety
 $\ini_\tau V$ with $\sigma\subsetneq\tau$, which contradicts the maximality of $\sigma$.
\end{proof}

The proof of the general case follows the proof for hypersurfaces.

\begin{lemma}
 Suppose that $V\subset(\C^\times)^n$ has pure codimension $k{+}1$.
 Let $\pi\subset\T^n$ be an affine subgroup of dimension $k{+}1$ that is the coamoeba of
 an affine subgroup of $(\C^\times)^n$ that meets $V$ transversally and suppose that 
 $\gamma\in Z_k(\pi\smallsetminus\coscrA(V))$ is a nonzero cycle which bounds a positive 
 chain $\lambda\in\Pl_{k+1}(\pi)$ whose support meets $\coscrA(V)$.
 Then the support of $\lambda$ meets the phase limit set $\scrP^\infty(V)$ of $V$.
\end{lemma}

\begin{proof}
 Write \defcolor{$|\lambda|$} for the support of the chain $\lambda$.
 Since $\lambda$ is a positive chain, this is simply the union of the images of the
 simplices in $\lambda$, and the boundary of $|\lambda|$ is a subset of the support of
 $\partial\lambda$. 
 In particular, the boundary of $|\lambda|$ does not meet $\coscrA(V)$, but its relative interior
 $|\lambda|^\circ$ does.
 Applying a translation if necessary, we may assume that $\pi=\T_N$ and that $V$ meets
 $\C^\times_N$ transversally and at least one point in $V\cap\C^\times_N$ has argument in
 $|\lambda|^\circ$. 

 Suppose that $\scrP^\infty(V)\cap|\lambda|=\emptyset$.
 Then $V\cap\Arg^{-1}(|\lambda|^\circ)$ is compact, for otherwise there is an unbounded
 sequence in $V$ whose arguments lie in the compact set $|\lambda|$, and thus there is a
 subsequence whose arguments converge to a point of $\scrP^\infty(V)\cap|\lambda|$, a
 contradiction. 

 Write \defcolor{$S_\lambda$} for $\Arg^{-1}(|\lambda|)\cap\C^\times_N$, which is isomorphic
 to $\R^{k+1}_>\times|\lambda|$.
 Then
\[
   \Arg^{-1}(|\lambda|)\ \simeq\ 
   \R^{n-k-1}_>\times S_\lambda\ \simeq\
   \R^{n-k-1}_>\times\R^{k+1}_>\times|\lambda|\,.
\]
 Projecting $V\cap\Arg^{-1}(|\lambda|)$ to the first and second factors gives compact sets
 $K\subset\R^{n-k-1}_>$ and $K_N\subset\R^{k+1}_>$.
 Let $D\subset\R^{k+1}_>$ be a compact triangulated set whose interior contains $K_N$,
 so that $D\times|\lambda|$ is a compact subset of $\C^\times_N$ whose interior
 contains the projection of $V\cap\Arg^{-1}(|\lambda|)$ to $\C^\times_N$. 
 We may do this as the boundary of $\lambda$ is disjoint from $\coscrA(V)$ and the boundary
 of $D$ is disjoint from $K_N$.

 We have that $1\in K$ as $V$ meets $\C^\times_N$ in a point whose argument lies in
 $|\lambda|$. 
 Let $a\in \R^{n-k-1}\smallsetminus K$ and consider a path $[1,a]$ from $1$ to $a$ in
 $\R^{n-k-1}$. 
 We assume that $D$ is oriented so that the orientation of $D\times|\lambda|$ agrees with
 that of $\C^\times_N$.
 As $D$ is triangulated, we regard it as a chain which is the sum of the oriented
 simplices that cover it, each with coefficient 1.
 Then $\mu:=[1,a]\times (D\times\lambda)$ is a $(2k{+}3)$-chain whose boundary is 
 \begin{equation}\label{Eq:bdryRho}
   \partial\mu\ =\ 
     \bigl(\{1\}\times (D\times\lambda)\bigr)
    \ \bigcup\ 
    \bigl([1,a]\times \partial(D\times\lambda)\bigr)
    \ \bigcup\ 
    \bigl(\{a\}\times (D\times\lambda)\bigr)\,,
 \end{equation}
 where the factors $D\times\lambda$ in the first and third terms have opposite orientations.

 Consider the intersection number of the cycles $\partial\mu\cdot V$.
 As only the first component $(D\times\lambda)=\{1\}\times (D\times\lambda)$
 of $\partial\mu$ in~\eqref{Eq:bdryRho} meets $V$, 
 we have that $\partial\mu\cdot V=(D\times\lambda)\cdot V$.  
 Since $D\times\lambda$ is a positive chain of dimension $2k{+}2$ in $\C^\times_N$ and $V$ is a
 complex subvariety of $(\C^\times)^n$ of complex codimension $k{+}1$, they meet
 transversally. 
 Consequently, this intersection number is the sum over all points $p$ of 
 $(D\times|\lambda|)\cap V$ of the multiplicity of the chain $D\times\lambda$ at $p$.
 Each of these contributions is positive and the sum is nonempty, as we assumed that 
 $V$  meets $\C^\times_N$ in a point whose argument lies in $|\lambda|^\circ$.
 This is a contradiction, for the intersection number of $\partial\mu$ with $V$ is zero
 as $[\partial\mu]=0$ in $H_{2k+1}((\C^\times)^n)$.
\end{proof}

Virtually the same proof as that for Corollary~\ref{C:shell} gives a proof of
Lemma~\ref{L:crucial}. 

\section{Nonarchimedean coamoeba}\label{S:three}

We define nonarchimedean coamoebas and sketch their basic properties as given in~\cite{NSna}.
Suppose that $\K$ is a complete field with a nonarchimedean valuation $\nu$ and residue field $\C$.
That is, there is a surjective group homomorphism $\nu\colon\K^\times\to\Gamma$, where $\Gamma$ is a
totally ordered divisible group. 
If we define $\nu(0):=\infty$ to be greater than any element of $\Gamma$, then $\nu$ is a 
nonarchimedean valuation in that if $a,b\in \K$, we have 
$\nu(a{+}b)\geq \min\{\nu(a),\nu(b)\}$ with equality unless $\nu(a)=\nu(b)$. 
Finally, if we set 
\[
   R\ :=\ \{a\in \K\mid \nu(a)\geq 0\}
   \qquad\mbox{and}\qquad
   \mathfrak{m}\ :=\ \{a\in \K\mid \nu(a)>0\}\,,
\]
then $\mathfrak{m}$ is the maximal ideal of the local ring $R$ and $R/\mathfrak{m}\simeq\C$.

A subvariety $V\subset(\K^\times)^n$ has a \demph{nonarchimedean amoeba $\scrT(V)$} which is its image in
$\Gamma^n$ under the coordinatewise valuation map. 
For each point $w\in\Gamma^n$, there is a tropical reduction $\ini_wV$ of $V$, which 
is a subvariety of an $n$-dimensional complex torus.
Identifying this torus with $(\C^\times)^n$ requires the choice of a section of the valuation
homomorpism.
Making this choice, the \demph{nonarchimedean coamoeba $\nca(V)$} of $V$ is the image of these tropical
reductions under the argument map to $\T^n$, and is therefore the union of the coamoebas of all tropical
reductions of $V$.
\begin{example}\label{Ex:tropical_space_line}
 Let $\zeta$ be a primitive third root of unity, $\omega:=1+\sqrt{-1}$,
 and $t$ be an element of $\K^\times$ with valuation 1.
 Consider the line $\ell\subset(\K^\times)^3$ defined by
\[
   x+\zeta y + \zeta^2 t\ =\
   \sqrt{-1}\cdot x +z - \omega\ =\ 0\,.
\]
 We display its nonarchimedean amoeba, assuming that $\Gamma\subset\R$.
\[
  \begin{picture}(255,95)(-66,-1)
   \put(0,0){\includegraphics{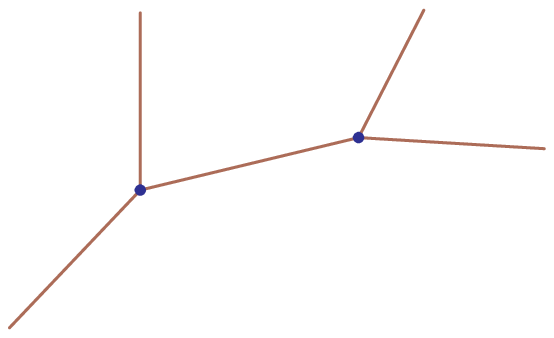}}
   \put(-6,40){$(0,0,0)$}    \put(82,40){$(1,1,0)$} 
   \put(-1,85){$(0,0,s)$}     \put(122,85){$(1,1{+}s,0)$}
   \put(-66,0){$(-s,-s,-s)$} \put(140,40){$(1{+}s,1,0)$}
  \end{picture}
\]
This has two vertices $(0,0,0)$ and $(1,1,0)$ connected by a line segment
($(s,s,0)$ for $0<s<1$) and four rays as indicated for $s>0$.
For each face of this polyhedral complex, $\ell$ has a different tropical reduction.
Figure~\ref{F:NA} shows two views of $\nca(\ell)$, which is the union of 
seven coamoebas, one for each tropical reduction of $\ell$.
\begin{figure}[htb]
  \includegraphics[height=150pt]{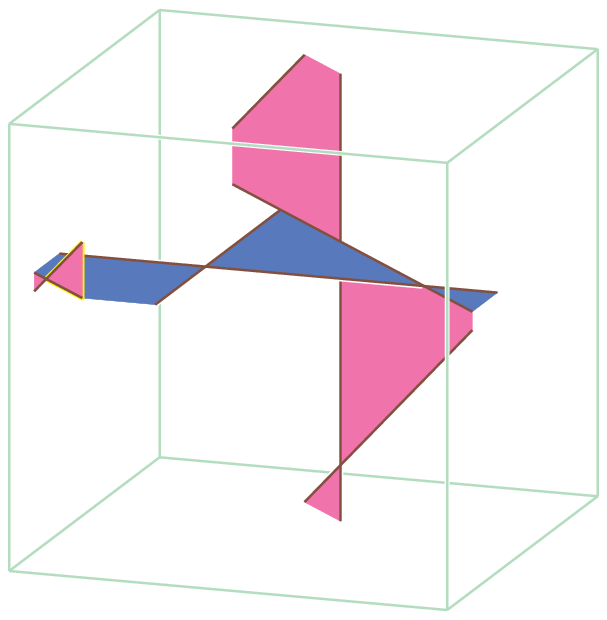}   \qquad
  \includegraphics[height=165pt]{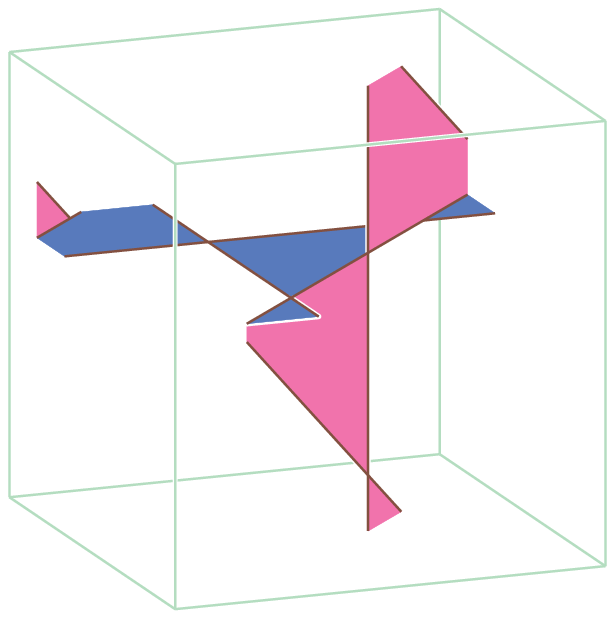} 
\caption{Non-archimedean coamoeba of a line}
\label{F:NA}
\end{figure}
It is the closure of the two corresponding to the vertices of $\scrT(\ell)$, each of
which is a coamoeba of a line in a plane consisting of two triangles.
Here $(0,0,0)$ corresponds to the vertical triangles and $(1,1,0)$ to the 
horizontal triangles.
These two coamoebas are attached along the coamoeba of the tropical reduction corresponding
to the edge between the vertices, and each has two boundary coamoebas corresponding to the
unbounded rays at each vertex.
\qed
\end{example}

The nonarchimedean coamoeba is a subset of $\T^n$.
Let $\defcolor{\lnca(V)}\subset\R^n$ be the lifted nonarchimedean coamoeba, its pullback to the
universal cover $\R^n$ of $\T^n$.

The tropical variety $\scrT(V)$ admits a structure of a polyhedral complex $\Sigma$ that
is compatible with tropical reductions in that if $\sigma\in\Sigma$ is a cone, then the tropical
reductions $\ini_wV$ are equal for all $w$ in the relative interior of $\sigma$.
Write \defcolor{$\ini_\sigma V$} for this common tropical reduction.
The structure of nonarchimedean coamoebas was described in~\cite{NSna}.

\begin{proposition}\label{P:ncoa}
 Let $V\subset(\K^\times)^n$ be a subvariety and fix a compatible polyhedral structure on its
 nonarchimedean amoeba $\scrT(V)$.
 The nonarchimedean coamoeba $\nca(V)$ is closed in $\T^n$.
 It is the (finite) union of archimedean coamoebas $\coscrA(\ini_\sigma V)$, for $\sigma\in\Sigma$.
 It is the closure of the union of archimedean coamoebas corresponding to the minimal faces of $\Sigma$.
 The same is true for the lifted nonarchimedean coamoeba, $\lnca(V)$.
\end{proposition}

We state and prove our result on the higher convexity of complements of lifted nonarchimedean coamoebas.

\begin{theorem}
 Let $V\subset(\K^\times)^n$ be a subvariety of pure codimension $k{+}1$.
 Then the complement of its lifted nonarchimedean coamoeba is $k$ convex.
\end{theorem}

\begin{proof}
 Choosing a compatible polyhedral complex as in Proposition~\ref{P:ncoa}, we have
\[
   \R^n\smallsetminus\lnca(V)\ =\ 
     \bigcap_{\substack{\sigma\in\Sigma\\\sigma\mbox{\scriptsize minimal}}}
     \bigl(\R^n\smallsetminus\lcoscrA(\ini_\sigma V)\bigr)\,,
\]
 which is a finite intersection.
 Since each complement $\R^n\smallsetminus\lcoscrA(\ini_\sigma V)$ is $k$-convex,
 Proposition~\ref{P:intersection} implies that $\R^n\smallsetminus\lnca(V)$ is $k$-convex. 
\end{proof}

\def\cprime{$'$}
\providecommand{\bysame}{\leavevmode\hbox to3em{\hrulefill}\thinspace}
\providecommand{\MR}{\relax\ifhmode\unskip\space\fi MR }
\providecommand{\MRhref}[2]{%
  \href{http://www.ams.org/mathscinet-getitem?mr=#1}{#2}
}
\providecommand{\href}[2]{#2}


\end{document}